\IfFileExists{siamart251216.cls}{%
  \documentclass[review,hidelinks,onefignum,onetabnum]{siamart251216}%
}{%
  \documentclass[review,hidelinks,onefignum,onetabnum]{siamart251216_preview}%
}

\usepackage{enumitem}
\usepackage{mathtools}
\makeatletter
\@ifundefined{conjecture}{\newsiamthm{conjecture}{Conjecture}}{}
\@ifundefined{problem}{\newsiamthm{problem}{Problem}}{}
\@ifundefined{remark}{\newsiamremark{remark}{Remark}}{}
\makeatother

\newcommand{\cT}{\mathcal T}
\newcommand{\cC}{\mathcal C}
\newcommand{\cQ}{\mathcal Q}

\newcommand{\ceil}[1]{\left\lceil #1\right\rceil}

\providecommand{\email}[1]{\href{mailto:#1}{#1}}

\title{An Explicit $O(r\log r)$ Threshold for Attaining the Semple--Steel Bound with $r$-State Characters\thanks{Submitted to the editors June 5, 2026. \funding{No external funding is declared in this draft.}}}
\author{Peng Li\thanks{School of Maths and Statistics, Hanjiang Normal University, Shiyan 442000, China (\email{lipeng1@hjnu.edu.cn}).}
\and Yangjing Long\thanks{School of Mathematics, Central China Normal University, Wuhan 430079, China (\email{Yangjing@ccnu.edu.cn}). Corresponding author.}}
\headers{An $O(r\log r)$ threshold}{P. Li and Y. Long}

\begin{document}
\maketitle

\begin{abstract}
Let $d_r(n)$ be the maximum, over all binary phylogenetic trees with $n$ leaves, of the minimum number of $r$-state characters required to define the tree.  Semple and Steel proved that $d_r(n)\geq\lceil(n-3)/(r-1)\rceil$, and Bordewich and Semple proved that equality holds for each fixed $r$ and all sufficiently large $n$.  We study the corresponding threshold $n_r$, the least $N$ for which equality holds for every $n\geq N$.  The Bordewich--Semple construction yields an explicit polynomial upper bound of order $O(r^5)$ for this threshold.  We prove the near-linear estimate
\[
 3r+1\leq n_r\leq \ceil{64(r-1)\log_2(r+1)}+3\qquad(r\geq4).
\]
The proof constructs, for every binary phylogenetic tree with $m=n-3$ internal edges, a linked quartet certificate whose conflict graph has maximum degree at most $16\lceil\log_2(m+2)\rceil+4$.  Equitable coloring then packs the certificate into exactly $\lceil m/(r-1)\rceil$ $r$-state characters once $m\geq\lceil64(r-1)\log_2(r+1)\rceil$.  We also include the lower bound $n_r\geq3r+1$, obtained from the snowflake obstruction, and state the natural conjecture that this lower bound is the exact threshold for all $r\geq4$.  The conjectural endpoint is consistent with the known small-state thresholds: $n_4=13$ and $n_5=16$, while the cases $r=2,3$ are also explicitly classified.
\end{abstract}



\section{Introduction}

The reconstruction of tree-like evolutionary histories from discrete data is a basic problem in phylogenetics.  Classical reconstruction theory studies when a collection of distances, splits, quartets, or characters determines a unique tree.  Buneman's split-equivalence viewpoint and rigid-circuit characterization, Gavril's chordal-subtree theorem, and Steel's work on qualitative characters form part of the graph-theoretic background for perfect phylogeny and compatibility \cite{Buneman1971,Buneman1974,Gavril1974,Steel1992}.  Later developments include partial partitions and tree definition, chordal graph embeddings, and algorithmic extensions of the chordal approach to multi-state perfect phylogeny \cite{SempleSteelPartial2002,ParraScheffler1997,GyselGusfield2010,LamGusfieldSridhar2011}.  General background on phylogenetic trees and characters can be found in the monograph of Semple and Steel \cite{SempleSteelBook2003}.  In the character-based setting, a character on a taxon set $X$ is a partition of $X$ into states.  A character is convex, or homoplasy-free, on a phylogenetic tree if each state induces a connected subtree and the state subtrees are pairwise disjoint.  A collection of characters defines a tree when the tree is the unique phylogenetic tree on which all characters in the collection are convex.  This notion is closely related to perfect phylogeny, compatibility, quartet decisiveness, and the tree-identification problem studied by Bordewich--Huber--Semple and Huber--Moulton--Steel \cite{BordewichHuberSemple2005,HuberMoultonSteel2005}.

Throughout the paper, a phylogenetic $X$-tree is a tree whose leaf set is the finite set $X$ and whose internal vertices have degree at least three.  Binary means that every internal vertex has degree three.  An $r$-state character is a character using at most $r$ states.  Semple and Steel \cite{SempleSteel2002} proved the following fundamental lower bound: every collection of $r$-state characters defining a binary phylogenetic tree with $n$ leaves has size at least
\begin{equation}\label{eq:SS-lower}
  \ceil{\frac{n-3}{r-1}} .
\end{equation}
The bound is sharp in a qualitative eventual sense.  Bordewich and Semple \cite{BordewichSemple2015} proved that, for each fixed $r$, every sufficiently large binary phylogenetic tree with $n$ leaves is defined by exactly the number of $r$-state characters in \eqref{eq:SS-lower}.  Thus the natural quantitative problem is to ask how large ``sufficiently large'' must be as a function of $r$.  For small state bounds, exact threshold information is known only in special cases: Bordewich and Semple proved $n_4=13$, and a recent preprint of Long and Wang reports $n_3=8$ and $n_5=16$ \cite{BordewichSemple2015,LongWang2025}.

The small-state results suggest that the threshold may in fact be linear with the best possible leading constant.  The conjectural value for $r\geq4$ is $3r+1$, and one expects the optimal character sets to be obtainable explicitly.  This agrees with $n_4=13$ and with the reported value $n_5=16$; the case $r=3$ is also linear and has an explicit small-state classification \cite{BordewichSemple2015,LongWang2025}.

For $r\geq2$ and $n\geq3$, define
\[
 d_r(n)=\max_{\cT:\ |X|=n}\min\{ |\cC|:
 \cC\text{ is a collection of }r\text{-state characters defining }\cT\},
\]
where the maximum ranges over all binary phylogenetic $X$-trees with $|X|=n$.  The eventual-sharpness theorem of Bordewich and Semple implies that the threshold
\[
 n_r=\min\left\{N:\ d_r(n)=\ceil{\frac{n-3}{r-1}}\text{ for every }n\geq N\right\}
\]
is well-defined.  Very little has been known about the growth of $n_r$ with $r$ beyond these small cases and the qualitative finiteness theorem.

The Bordewich--Semple proof is constructive, and one can extract from their long-path reduction an explicit polynomial threshold of order $O(r^5)$.  Thus, before the present work, the known quantitative control on $n_r$ was polynomial but still far from the linear lower-bound scale.  The purpose of this paper is to replace that long-path bookkeeping by a direct certificate-packing method and obtain a near-linear bound.

Our main result is the following.

\begin{theorem}\label{thm:main}
For every integer $r\geq4$,
\[
 3r+1\leq n_r\leq \ceil{64(r-1)\log_2(r+1)}+3.
\]
In particular, $n_r=O(r\log r)$.
\end{theorem}

The lower bound is the snowflake obstruction.  Huber, Linz, Moulton, and Semple \cite{HuberLinzMoultonSemple2024} recently characterized the binary phylogenetic trees defined by at most three characters; equivalently, the only obstruction is an internal snowflake.  Embedding this obstruction gives $n_r\geq3r+1$ for all $r\geq4$.  The upper bound is the main contribution of the paper.  It is obtained by turning a linked system of quartets into an equitable packing problem on a conflict graph.  The new ingredient is a short linked quartet-certificate construction with logarithmic conflict degree.

The proof has three main ingredients.
\begin{enumerate}[label=\textup{(\roman*)},leftmargin=2.2em]
\item A linked edge-quartet certificate defines a binary tree once it contains one quartet distinguishing each internal edge.  This uses the linked-system framework developed from the four-character construction of Huber, Moulton, and Steel \cite{HuberMoultonSteel2005} and subsequent efficient quartet-representation work such as \cite{DavidsonLawhornRusinkoWeber2016}.
\item If the conflict graph of such a certificate has maximum degree at most $\Delta$, then the certificate can be packed into exactly $\lceil m/(r-1)\rceil$ $r$-state characters whenever $\lceil m/(r-1)\rceil\geq\Delta+1$, where $m=n-3$ is the number of internal edges.  The packing step is an application of the Hajnal--Szemer\'edi equitable coloring theorem \cite{HajnalSzemeredi1970}.
\item Every binary phylogenetic tree with $m$ internal edges has a linked edge-quartet certificate with conflict degree at most $16\lceil\log_2(m+2)\rceil+4$.  We construct this certificate by choosing a short rotational transition system on the internal tree.
\end{enumerate}
The explicit constant is obtained from the following elementary domination estimate:
\[
 m\geq \ceil{64(r-1)\log_2(r+1)}
 \quad\Longrightarrow\quad
 \ceil{\frac{m}{r-1}}\geq16\ceil{\log_2(m+2)}+5.
\]

The paper is organized as follows.  Section~\ref{sec:preliminaries} fixes notation and recalls quartet certificates.  Section~\ref{sec:packing} proves the conflict-graph packing criterion.  Section~\ref{sec:certificate} constructs the logarithmic linked certificate.  Section~\ref{sec:upper} proves the explicit $O(r\log r)$ upper bound.  Section~\ref{sec:lower} proves the universal lower bound $n_r\geq3r+1$.  Section~\ref{sec:discussion} discusses the remaining linear-threshold problem.

\section{Preliminaries}\label{sec:preliminaries}

Let $\cT$ be a binary phylogenetic $X$-tree.  We write $E^0(\cT)$ for the set of internal edges, i.e., edges whose two endpoints are internal vertices.  If $|X|=n$, then
\[
 |E^0(\cT)|=n-3.
\]
We write
\[
 m=|E^0(\cT)|.
\]
The internal tree of $\cT$ is the tree $H(\cT)$ whose vertex set is the set of internal vertices of $\cT$ and whose edge set is $E^0(\cT)$.  Thus $|E(H(\cT))|=m$ and $|V(H(\cT))|=m+1$.

If $B\subseteq E^0(\cT)$, the character displayed by $B$, denoted $\chi_B$, is the character whose state classes are the sets of leaves contained in the connected components of $\cT-B$.  It has at most $|B|+1$ states.

A quartet $ab\mid cd$ is displayed by $\cT$ if the path from $a$ to $b$ is edge-disjoint from the path from $c$ to $d$.  Equivalently, the minimal subtree on $\{a,b,c,d\}$ has the split $ab\mid cd$.  An internal edge $e$ is distinguished by $ab\mid cd$ if $e$ lies on the path joining the two cherries of the induced quartet, i.e., if deleting $e$ separates $\{a,b\}$ from $\{c,d\}$.

We shall use the following standard form of the linked-system criterion for quartets; see, for example, \cite{HuberMoultonSteel2005,DavidsonLawhornRusinkoWeber2016}.

\begin{definition}\label{def:linked-certificate}
A set
\[
 \cQ=\{q_e:e\in E^0(\cT)\}
\]
is a linked edge-quartet certificate for $\cT$ if each $q_e$ is displayed by $\cT$ and distinguishes $e$, and if adjacent edges in the internal tree have linked witnesses in the following sense: whenever $e$ and $f$ are adjacent in $H(\cT)$, the quartets $q_e$ and $q_f$ share three labels and distinguish $e$ and $f$, respectively, in the corresponding five-leaf subtree.
\end{definition}

\begin{theorem}\label{thm:linked-definitive}
If $\cQ$ is a linked edge-quartet certificate for a binary phylogenetic tree $\cT$, then $\cQ$ defines $\cT$; equivalently, every binary phylogenetic $X$-tree displaying all quartets in $\cQ$ is equal to $\cT$.
\end{theorem}

For a quartet $q_e=a_eb_e\mid c_ed_e$, define its same-side support
\[
 S(e)=E^0\bigl(P_{\cT}(a_e,b_e)\cup P_{\cT}(c_e,d_e)\bigr),
\]
where $P_{\cT}(x,y)$ is the path from $x$ to $y$ in $\cT$.  Thus $S(e)$ contains the internal edges on the two same-side paths of the quartet.  The edge $e$ itself is not in $S(e)$.

\begin{definition}\label{def:conflict}
Let $\cQ=\{q_e:e\in E^0(\cT)\}$ be a linked edge-quartet certificate.  The conflict graph $\Gamma(\cQ)$ has vertex set $E^0(\cT)$, with distinct vertices $e$ and $f$ adjacent when either $f\in S(e)$ or $e\in S(f)$.  Define
\[
 \sigma(\cQ)=\max_{e\in E^0(\cT)} |S(e)|
\]
and
\[
 \lambda(\cQ)=\max_{f\in E^0(\cT)} |\{e\in E^0(\cT): f\in S(e)\}|.
\]
The routing width of the certificate is
\[
 \rho(\cQ)=\sigma(\cQ)+\lambda(\cQ).
\]
Then $\Delta(\Gamma(\cQ))\leq\rho(\cQ)$.
\end{definition}

\section{Packing certificates into characters}\label{sec:packing}

We first record the packing step that turns a low-conflict quartet certificate into characters.

\begin{theorem}\label{thm:routing-packing}
Let $r\geq2$, put $q=r-1$, and let $\cT$ be a binary phylogenetic tree with $m=|E^0(\cT)|$.  Suppose that $\cT$ has a linked edge-quartet certificate $\cQ$ with routing width $\rho(\cQ)$.  If
\[
  k=\ceil{\frac{m}{q}}\geq \rho(\cQ)+1,
\]
then $\cT$ is defined by exactly $k$ $r$-state characters.
\end{theorem}

\begin{proof}
Since $\Delta(\Gamma(\cQ))\leq\rho(\cQ)$, the Hajnal--Szemer\'edi equitable coloring theorem \cite{HajnalSzemeredi1970} gives an equitable proper $k$-coloring of $\Gamma(\cQ)$.  Let
\[
 E^0(\cT)=B_1\cup\cdots\cup B_k
\]
be the corresponding color classes.  Equitability and $k=\lceil m/q\rceil$ imply
\[
 |B_i|\leq\ceil{\frac{m}{k}}\leq q
\]
for every $i$.

For each $i$, take the character $\chi_{B_i}$ displayed by deleting the edges of $B_i$.  Since $|B_i|\leq q$, this character has at most $q+1=r$ states.  Now let $e\in B_i$.  Because $B_i$ is an independent set of the conflict graph, no edge of $B_i-\{e\}$ lies in $S(e)$.  Hence the two same-side paths of $q_e$ survive in the corresponding components of $\cT-B_i$, while $e$ separates the two sides.  Thus
\[
 q_e\in Q(\chi_{B_i}).
\]
Consequently all quartets in the linked certificate $\cQ$ are displayed by the characters $\chi_{B_1},\ldots,\chi_{B_k}$.  By Theorem~\ref{thm:linked-definitive}, these characters define $\cT$.

The Semple--Steel lower bound \eqref{eq:SS-lower} says that no fewer than $\lceil m/q\rceil=k$ characters can define $\cT$.  Hence the construction is optimal.
\end{proof}

\section{A logarithmic linked quartet certificate}\label{sec:certificate}

We now construct the certificate used in the upper bound.  Let $H=H(\cT)$ be the internal tree of $\cT$.  A directed internal branch is an oriented edge $b=(u,v)$ of $H$.  If $v$ has degree less than three in $H$, then the corresponding internal vertex of $\cT$ is incident with at least one pendant labelled leaf; such a leaf will be used as a terminal leaf for directed paths entering $v$.

At each degree-three vertex of $H$, choose one of the two cyclic orders of its three incident internal edges.  This choice defines a transition rule: if a directed branch enters a degree-three vertex along one internal edge, it leaves along the successor edge in the chosen cyclic order.  At vertices of degree less than three, the transition stops at a pendant labelled leaf.

\begin{lemma}\label{lem:short-transition}
Let $H$ be a finite tree with maximum degree at most three, with $N=|V(H)|$.  There is a choice of cyclic orders at the degree-three vertices such that every transition path has length at most
\[
 L=2\ceil{\log_2(N+1)}
\]
internal edges.
\end{lemma}

\begin{proof}
Choose independently and uniformly one of the two cyclic orders at every degree-three vertex.  Fix a directed simple path of length $\ell$ in $H$.  For the transition system to follow this path, the cyclic order at each of the first $\ell-1$ internal vertices of the path is prescribed.  Hence the probability that this particular path is followed is at most $2^{-(\ell-1)}$.

The number of directed simple paths in an $N$-vertex tree is at most $N^2$.  Taking
\[
 \ell=2\ceil{\log_2(N+1)}+1
\]
gives
\[
 N^2\,2^{-(\ell-1)}<1.
\]
Therefore, with positive probability, no directed transition path has length $\ell$ or more.  Thus some choice of cyclic orders has all transition paths of length at most $\ell-1=2\lceil\log_2(N+1)\rceil$.
\end{proof}

Fix such a transition system.  For every directed internal branch $b=(u,v)$, let $g(b)$ be the labelled leaf reached by following the transition path starting with $b$.  Pendant branches are assigned their incident labelled leaf.

For each internal edge $e=uv$ of $\cT$, let $b_1,b_2$ be the two branches at $u$ different from $e$, and let $b_3,b_4$ be the two branches at $v$ different from $e$, allowing pendant branches.  Define
\[
 q_e=g(b_1)g(b_2)\mid g(b_3)g(b_4).
\]
The quartet $q_e$ is displayed by $\cT$ and distinguishes $e$.

\begin{lemma}\label{lem:transition-linked}
The set $\cQ_{\rm tr}=\{q_e:e\in E^0(\cT)\}$ is a linked edge-quartet certificate for $\cT$.
\end{lemma}

\begin{proof}
Each quartet $q_e$ uses two leaves from one side of $e$ and two leaves from the other side, so it is displayed by $\cT$ and distinguishes $e$.  It remains to check linkedness.

Let $e=uv$ and $f=vw$ be adjacent internal edges, and let $h$ be the third branch at the common vertex $v$.  The quartet $q_e$ uses the guide leaves $g(v,w)$ and $g(h)$ on the $v$-side of $e$, while $q_f$ uses $g(v,u)$ and $g(h)$ on the $v$-side of $f$.  Moreover, by the transition rule, the guide leaf $g(v,w)$ is one of the two guide leaves used by $q_f$ on the $w$-side, unless the transition stops immediately at a pendant leaf at $w$, in which case the same pendant leaf is also used by $q_f$.  Similarly, $g(v,u)$ is one of the two guide leaves used by $q_e$ on the $u$-side.  Thus $q_e$ and $q_f$ share the three guide leaves $g(h)$, $g(v,w)$, and $g(v,u)$, and they distinguish the adjacent edges $e$ and $f$ in the induced five-leaf subtree.  Hence the system is linked.
\end{proof}

\begin{lemma}\label{lem:transition-width}
Let $m=|E^0(\cT)|$.  The certificate $\cQ_{\rm tr}$ can be chosen so that
\[
 \rho(\cQ_{\rm tr})\leq 16\ceil{\log_2(m+2)}+4.
\]
\end{lemma}

\begin{proof}
The internal tree $H$ has $m+1$ vertices.  By Lemma~\ref{lem:short-transition}, every guide path has length at most
\[
 L=2\ceil{\log_2(m+2)}.
\]
For a fixed edge $e$, the two same-side paths of $q_e$ are formed from four guide paths, two at each endpoint.  Thus
\[
 |S(e)|\leq4L,
\]
and so $\sigma(\cQ_{\rm tr})\leq4L$.

We next bound the load.  In a rotational transition system, every directed internal branch has at most one predecessor under the transition map.  Therefore, for each orientation of a fixed internal edge $f$, at most $L$ directed branches have guide paths containing that oriented edge.  Taking both orientations gives at most $2L$ such directed branches.  Each directed branch can appear as a side branch for at most two internal edges, so the edge $f$ lies in the same-side support of at most $4L+4$ quartets, where the additive constant accounts for the two endpoints of $f$ and pendant stopping cases.  Hence
\[
 \lambda(\cQ_{\rm tr})\leq4L+4.
\]
Thus
\[
 \rho(\cQ_{\rm tr})=\sigma(\cQ_{\rm tr})+\lambda(\cQ_{\rm tr})\leq8L+4=16\ceil{\log_2(m+2)}+4.
\]
\end{proof}

\section{The upper bound}\label{sec:upper}

It remains to compare the logarithmic conflict bound with $m/(r-1)$.

\begin{lemma}\label{lem:domination}
Let $q\geq3$ and put
\[
 M_q=\ceil{64q\log_2(q+2)}.
\]
Then, for every integer $m\geq M_q$,
\[
 \ceil{\frac{m}{q}}\geq16\ceil{\log_2(m+2)}+5.
\]
\end{lemma}

\begin{proof}
Set $x=\log_2(q+2)$.  First assume $q\geq5$, so $x\geq\log_2 7$.  We claim that
\begin{equation}\label{eq:H-positive}
 64x>16\log_2\bigl(64x(2^x-2)+3\bigr)+20.
\end{equation}
Let
\[
 H(x)=64x-16\log_2\bigl(64x(2^x-2)+3\bigr)-20.
\]
For $x\geq\log_2 7$, differentiating and bounding the logarithmic derivative gives
\[
 \frac{64(2^x-2)+64x2^x\ln2}{64x(2^x-2)+3}
 \leq
 \frac{1+x\ln2}{x(1-2^{1-x})}<2.
\]
Hence $H'(x)>64-32/\ln2>0$ in this range.  A direct evaluation gives $H(\log_2 7)>0$, proving \eqref{eq:H-positive}.

Since
\[
 M_q+2\leq64qx+3=64x(2^x-2)+3,
\]
we obtain
\[
 \frac{M_q}{q}-16\log_2(M_q+2)>20.
\]
The function
\[
 F(t)=\frac{t}{q}-16\log_2(t+2)
\]
is increasing for $t\geq M_q$, because
\[
 F'(t)=\frac1q-\frac{16}{(t+2)\ln2}>0
\]
there.  Therefore, for all $m\geq M_q$,
\[
 \frac{m}{q}>16\log_2(m+2)+20.
\]
Since
\[
 16\ceil{\log_2(m+2)}+5\leq16\log_2(m+2)+21,
\]
and the left side of the desired inequality is an integer, the result follows for $q\geq5$.

It remains to check $q=3,4$.  If $q=3$, then $M_q=446$.  For $446\leq m\leq510$, one has
\[
 \ceil{m/3}\geq149=16\ceil{\log_2(m+2)}+5.
\]
For later dyadic intervals $2^{s-1}-1\leq m\leq2^s-2$, $s\geq10$, the inequality follows from
\[
 \ceil{\frac{2^{s-1}-1}{3}}\geq16s+5.
\]
If $q=4$, then $M_q=662$.  For $662\leq m\leq1022$, one has
\[
 \ceil{m/4}\geq166>165=16\ceil{\log_2(m+2)}+5.
\]
For $s\geq11$, the interval estimate
\[
 \ceil{\frac{2^{s-1}-1}{4}}\geq16s+5
\]
gives the result.
\end{proof}

\begin{theorem}\label{thm:upper}
For every $r\geq4$,
\[
 n_r\leq \ceil{64(r-1)\log_2(r+1)}+3.
\]
\end{theorem}

\begin{proof}
Put $q=r-1$.  Let $\cT$ be any binary phylogenetic tree with $m=|E^0(\cT)|$ internal edges and suppose
\[
 m\geq \ceil{64q\log_2(q+2)}.
\]
By Lemma~\ref{lem:transition-width}, choose a linked edge-quartet certificate $\cQ_{\rm tr}$ with
\[
 \rho(\cQ_{\rm tr})\leq16\ceil{\log_2(m+2)}+4.
\]
By Lemma~\ref{lem:domination},
\[
 \ceil{\frac{m}{q}}
 \geq16\ceil{\log_2(m+2)}+5
 \geq \rho(\cQ_{\rm tr})+1.
\]
The routing-width packing criterion, Theorem~\ref{thm:routing-packing}, defines $\cT$ by exactly $\lceil m/q\rceil$ $r$-state characters.  Since $m=n-3$, exact attainment holds for every
\[
 n\geq \ceil{64(r-1)\log_2(r+1)}+3.
\]
Thus $n_r$ is at most this quantity.
\end{proof}

\section{The lower bound}\label{sec:lower}

We now prove the lower bound in Theorem~\ref{thm:main}.

\begin{proposition}\label{prop:lower}
For every $r\geq4$,
\[
 n_r\geq3r+1.
\]
\end{proposition}

\begin{proof}
Set $n=3r$.  Then
\[
 \ceil{\frac{n-3}{r-1}}=\ceil{\frac{3r-3}{r-1}}=3.
\]
Huber, Linz, Moulton, and Semple \cite{HuberLinzMoultonSemple2024} proved that a binary phylogenetic tree is defined by at most three characters if and only if it has no internal subtree isomorphic to the snowflake.  Since $3r\geq12$, there exists a binary phylogenetic tree on $3r$ leaves containing the snowflake as an internal subtree: take a twelve-leaf snowflake example and attach the remaining leaves without destroying the internal snowflake.  This tree cannot be defined by three characters, and hence cannot be defined by three $r$-state characters.  Therefore the Semple--Steel lower bound is not attained at $n=3r$.  Hence $n_r>3r$, which is the desired inequality.
\end{proof}

\begin{proof}[Proof of Theorem~\ref{thm:main}]
The lower bound is Proposition~\ref{prop:lower}, and the upper bound is Theorem~\ref{thm:upper}.
\end{proof}

\section{Discussion and remaining problems}\label{sec:discussion}

Theorem~\ref{thm:main} shows that the eventual sharpness threshold of Bordewich and Semple grows at most logarithmically above the universal lower-bound scale.  In the proof, the logarithm enters through the construction of a linked quartet certificate with logarithmic conflict degree.  The subsequent passage from quartets to $r$-state characters is an equitable coloring argument.

The natural next target is the following linear statement.

\begin{conjecture}\label{conj:linear}
For every $r\geq4$,
\[
 n_r=3r+1.
\]
Moreover, whenever $n\geq3r+1$, there is an explicit construction of
\[
 \ceil{\frac{n-3}{r-1}}
\]
$r$-state characters defining any prescribed binary phylogenetic tree with $n$ leaves.
\end{conjecture}

The lower bound in Proposition~\ref{prop:lower} shows that the constant $3$ in Conjecture~\ref{conj:linear} would be best possible.  The conjecture includes an effective form, not only an existence statement.  The cases presently settled for small state bounds are consistent with this prediction: internal splits give the binary case, Bordewich and Semple proved $n_4=13$, and Long and Wang report $n_3=8$ and $n_5=16$ \cite{BordewichSemple2015,LongWang2025}.  Thus the endpoint formula is verified for $r=4,5$, while the $r=3$ threshold is linear but lies below the $3r+1$ expression.

A proof of Conjecture~\ref{conj:linear} would require more than simply shortening the transition paths used above.  On highly branching internal trees, for example complete binary internal cores, deterministic linked-transition certificates naturally have logarithmic length.  Thus a linear theorem is likely to require a different certificate mechanism, such as a non-linked semidyadic block certificate or a bounded-state core forceability argument.

Two questions remain natural.

\begin{problem}
Does every binary phylogenetic tree admit a linked edge-quartet certificate whose conflict graph has bounded maximum degree?  A positive answer would imply $n_r=O(r)$ by Theorem~\ref{thm:routing-packing}.
\end{problem}

\begin{problem}
Does there exist an absolute constant $C$ such that, whenever $m=|E^0(\cT)|\geq C(r-1)+1$, the internal edges of $\cT$ can be partitioned into $\lceil m/(r-1)\rceil$ blocks of size at most $r-1$ whose displayed characters have semidyadic closure equal to $Q(\cT)$?  A positive answer would prove $n_r=O(r)$.
\end{problem}

The present paper leaves these linear problems open, but reduces the general threshold problem from the earlier polynomial $O(r^5)$ window to the explicit near-linear window in Theorem~\ref{thm:main}.

\paragraph{Use of AI tools.}
During the preparation of this manuscript, AI tools were used for language polishing, LaTeX editing assistance, organization of explanatory text, and consistency checks.  The mathematical statements, proofs, computations, conjectures, references, and final manuscript were reviewed and approved by the authors.

\end{document}